\documentclass[12pt]{article}
\usepackage[T1]{fontenc}
\usepackage{amsmath,amssymb,amsthm}
\usepackage[margin=1.1in,headheight=15pt]{geometry}
\usepackage{setspace}
\usepackage{fancyhdr}
\usepackage{booktabs}
\usepackage{microtype}
\usepackage{url}
\usepackage[hidelinks]{hyperref}
\fancypagestyle{plain}{%
  \fancyhf{}%
  \fancyhead[R]{\thepage}%
}
\newtheorem{theorem}{Theorem}[section]
\newtheorem{lemma}[theorem]{Lemma}
\newtheorem{corollary}[theorem]{Corollary}
\newtheorem{proposition}[theorem]{Proposition}
\theoremstyle{definition}
\newtheorem{definition}[theorem]{Definition}
\theoremstyle{remark}

\numberwithin{equation}{section}

\newcommand{\ind}{\mathbf 1}
\DeclareMathOperator{\rad}{rad}
\DeclareMathOperator{\out}{out}
\DeclareMathOperator{\inn}{in}
\DeclareMathOperator{\wt}{wt}
\hypersetup{pdftitle={Surjectivity of the Enots--Wolley Sequence},
  pdfauthor={Nathan Myles Nichols},pdfsubject={Surjectivity of the Enots--Wolley sequence}}
\title{Surjectivity of the Enots--Wolley Sequence}
\author{Nathan Myles Nichols\\[0.3em]
  \normalsize Independent Researcher\\[0.2em]
  \normalsize\texttt{nathan.m.nichols@proton.me}}
\date{}
\begin{document}
\maketitle
\begin{abstract}
We prove that the Enots--Wolley sequence, a greedy sequence defined by local
divisibility and coprimality conditions, contains every positive integer with
at least two distinct prime divisors. This resolves an open conjecture by
Shannon and Sloane. The only analytic number-theoretic inputs are the prime
number theorem and Mertens' estimate for reciprocal primes.
\end{abstract}
\medskip
\section{Introduction}
\label{sec:introduction}

The Enots--Wolley sequence, introduced by Shannon and Sloane as OEIS A336957,
is the lexicographically earliest infinite sequence of distinct positive
integers beginning with $1,2$ in which each later term shares a prime divisor
with its predecessor and is coprime to the term two places earlier
\cite{oeis-ew}. Its first terms are
\[
1,2,6,15,35,14,12,33,55,10,18,21,77,\ldots.
\]
The name reverses ``Yellowstone'': the Yellowstone permutation exchanges the
roles of the preceding two terms in these divisibility conditions
\cite{yellowstone,oeis-yellowstone}.

The requirement that the sequence be infinite matters to its greedy definition.
The smallest unused integer satisfying the two divisibility conditions at a
given step need not admit an infinite continuation: after $1,2$, choosing $4$
would satisfy both conditions but leave no possible next term. For
Enots--Wolley, a choice admits infinite continuation precisely when it
introduces a prime divisor absent from its predecessor. This excludes $4$
and permits $6$.

Write $\omega(m)$ for the number of distinct prime divisors of $m$. Every
noninitial Enots--Wolley term has $\omega(m)\ge2$. The surjectivity conjecture
recorded with A336957 asserts that every such integer occurs
\cite[Conjecture 1]{oeis-ew}. We prove that conjecture.

\begin{theorem}[Surjectivity]\label{thm:main}
The Enots--Wolley sequence is a permutation of
\[
\{1,2\}\cup\{m\ge1:\omega(m)\ge2\}.
\]
\end{theorem}

\section{Preliminaries and elementary dynamics}
\label{sec:setup}

\subsection{Prime support and the local greedy rule}

For a positive integer $m$, let
\[
P(m)=\{r:r\text{ is prime and }r\mid m\},\qquad
\omega(m)=|P(m)|.
\]
Thus $P(m)$ is the \emph{prime support} of $m$, and $\omega(m)$ counts its
distinct prime divisors, without multiplicity. In particular,
$P(1)=\varnothing$ and $\omega(1)=0$. Put
\[
\mathcal A=\{m\ge1:\omega(m)\ge2\}.
\]

Starting from $a_1=1$, $a_2=2$, the local form of the Enots--Wolley rule takes
$a_n$ for $n\ge3$ to be the least unused member of $\mathcal A$ satisfying
\begin{align}
P(a_n)\cap P(a_{n-1})&\ne\varnothing
&&\text{(overlap)},\label{eq:overlap}\\
P(a_n)\cap P(a_{n-2})&=\varnothing
&&\text{(lag-two disjointness)},\label{eq:lagtwo}\\
P(a_n)\setminus P(a_{n-1})&\ne\varnothing
&&\text{(novelty)}.\label{eq:novelty}
\end{align}
We call an integer \emph{locally admissible} when it satisfies these support
conditions, whether or not it has already been used. An \emph{admissible
candidate} is a locally admissible integer which has not yet been selected.

The next observation explains why novelty is exactly the continuation
condition hidden in the lexicographically earliest infinite formulation.

\begin{lemma}[Infinite continuation criterion]
\label{lem:continuation}
Let $n\ge3$. Suppose $a_1,\ldots,a_{n-1}$ is a finite prefix beginning with
$1,2$, consisting of distinct terms, and satisfying overlap and lag-two
disjointness at every applicable index. Let $a_n$ be an unused positive
integer satisfying
\[
P(a_n)\cap P(a_{n-1})\ne\varnothing,
\qquad
P(a_n)\cap P(a_{n-2})=\varnothing.
\]
Then the extended prefix
\[
a_1,\ldots,a_n
\]
admits an infinite continuation satisfying the same two divisibility
conditions if and only if
\[
P(a_n)\setminus P(a_{n-1})\ne\varnothing.
\]
\end{lemma}
\begin{proof}
If $P(a_n)\subseteq P(a_{n-1})$, then every possible successor $a_{n+1}$
which overlaps $a_n$ also shares a prime with $a_{n-1}$. This contradicts the
lag-two disjointness condition required at index $n+1$, namely
\[
P(a_{n+1})\cap P(a_{n-1})=\varnothing.
\]
Hence no continuation exists.

Conversely, choose
\[
r\in P(a_n)\setminus P(a_{n-1}).
\]
Choose distinct primes $Q_1,Q_2,\ldots$ which divide none of
$a_1,\ldots,a_n$ and are different from $r$. Append
\[
rQ_1,\quad Q_1Q_2,\quad Q_2Q_3,\quad\ldots.
\]
The first appended term overlaps $a_n$ through $r$, is disjoint from
$a_{n-1}$, and introduces $Q_1$. The next term overlaps $rQ_1$ through
$Q_1$, is disjoint from $a_n$, and introduces $Q_2$. Thereafter each term
overlaps its predecessor through the newest prime, is disjoint from the term
two indices earlier, and introduces a fresh prime. All appended terms have two
distinct prime divisors and are distinct. Hence the continuation is infinite.
\end{proof}

Thus the extendible choices are exactly the unused integers satisfying
\eqref{eq:overlap}--\eqref{eq:novelty}; choosing the least such integer gives
the original lexicographically earliest infinite sequence
\cite[Theorem 1]{oeis-ew}. In particular, the sequence never gets stuck.
After the first greedy selection, novelty of the preceding term always supplies
an active prime
\[
r\in P(a_{n-1})\setminus P(a_{n-2}),
\]
and pairing $r$ with any globally unseen prime gives an unused candidate.

\begin{lemma}[Smaller admissible integers are historical]
\label{lem:smaller-historical}
At the selection of $a_n$, every locally admissible positive integer
$m<a_n$ has already occurred. Consequently, if an eligible integer $m$
never occurs, there is an index after which $m$ is never locally admissible.
\end{lemma}
\begin{proof}
The first assertion is the greedy rule. Each later index at which an omitted
$m$ is locally admissible would select a distinct integer below $m$. There
are only finitely many such integers.
\end{proof}

\subsection{Prime debuts and the unseen-prime frontier}
\label{sec:basic-dynamics}

A prime \emph{debuts} at its first occurrence as a divisor of a selected term.
It is \emph{unseen} before index $n$ if it divides none of
$a_1,\ldots,a_{n-1}$.

\begin{lemma}[A noninitial debut is a product of two distinct primes]
\label{lem:debut-pair}
If a prime $Q$ debuts in $H=a_n$ at an index $n\ge3$, then $H=bQ$ for a
unique prime $b$ dividing the predecessor. At most one previously unseen
prime debuts at a noninitial selection, and prime debuts occur in increasing
prime order.
\end{lemma}
\begin{proof}
Overlap gives a prime $b\in P(H)\cap P(a_{n-1})$. Since $H=a_n$ is lag-two
disjoint from $a_{n-2}$, we have $b\nmid a_{n-2}$. The prime $Q$ is unseen,
so $bQ$ is unused and locally admissible: $b$ supplies overlap, $Q$ supplies
novelty, and both avoid $a_{n-2}$. Thus $H\le bQ$. Conversely, $b$ and $Q$
are distinct prime divisors of $H$, so $bQ\mid H$ and $bQ\le H$. Therefore
$H=bQ$. Its only old prime is $b=H/Q$, and no second prime can debut there.

If a smaller prime $R<Q$ were still unseen, $bR$ would be a smaller unused
admissible candidate at the same selection. This proves increasing debut
order, including the initial appearance of $2$.
\end{proof}

Let $Q_N$ denote the least unseen prime \emph{after} the first $N$ terms,
and let $M_N=\max_{i\le N}a_i$.

\begin{lemma}[A fixed unseen prime bounds the entire prefix]
\label{lem:unseen-ceiling}
For every $N\ge2$,
\[
M_N<Q_N^2.
\]
Consequently every prime eventually appears.
\end{lemma}
\begin{proof}
At each earlier noninitial selection, choose an active predecessor prime
$r\in P(a_{n-1})\setminus P(a_{n-2})$. Increasing debut order gives
$r<Q_N$. The still-unseen prime $Q_N$ makes $rQ_N$ an unused admissible
candidate, so $a_n\le rQ_N<Q_N^2$. The seeds satisfy the same bound.
If $Q_N$ remained bounded, this would put infinitely many distinct selected
integers below a fixed bound. Hence $Q_N\to\infty$, and increasing debut
order shows that every prime appears.
\end{proof}

\subsection{Notation and classical prime estimates}
\label{sec:analytic}

All logarithms are natural. We write $f\ll_\theta g$, or $f=O_\theta(g)$,
when $|f|\le Cg$ with $C$ depending only on the indicated fixed parameters.
The notation $f\gg_\theta g$ reverses this comparison, $f\asymp_\theta g$
means both comparisons, and $f=o(g)$ means $f/g\to0$. For a fixed polynomial
range, $X^{o(1)}$ denotes a bound at most $X^\varepsilon$ for every fixed
$\varepsilon>0$ once $X$ is sufficiently large. We will identify the
parameters over which our error terms are uniform.

Let $\pi(x)$ count primes at most $x$, let $p_j$ be the $j$th prime, and put
\[
\vartheta(x)=\sum_{\substack{r\le x\\r\ \mathrm{prime}}}\log r.
\]
The prime number theorem and its equivalent prime-index form give
\begin{equation}
\pi(x)\sim\frac{x}{\log x},\qquad
\vartheta(x)\sim x,\qquad p_j\ll j\log(j+2).
\label{eq:prime-asymptotics}
\end{equation}
Mertens' estimate gives, for a constant $B_1$,
\begin{equation}
\sum_{\substack{r\le x\\r\ \mathrm{prime}}}\frac1r
=\log\log x+B_1+o(1).
\label{eq:mertens}
\end{equation}
See \cite[\S27.2, (27.2.3)--(27.2.4), and \S27.11, (27.11.8)]{dlmf}.
The assertion for $\vartheta$ follows from the one for $\pi$ by partial
summation.

For a positive integer $D$, write
\[
\rho(D)=\frac{\varphi(D)}{D}
=\prod_{r\mid D}\left(1-\frac1r\right),
\]
where $\varphi$ is Euler's totient function. Inclusion--exclusion gives,
uniformly for real $z>0$,
\begin{equation}
\#\{1\le u<z:(u,D)=1\}=z\rho(D)+O(2^{\omega(D)}).
\label{eq:coprime-count}
\end{equation}
Here and below $(u,D)$ denotes the greatest common divisor.

\begin{lemma}[Polynomially bounded moduli]\label{lem:polynomial-modulus}
Fix $C>0$. Uniformly for positive integers $D\le X^C$, as $X\to\infty$,
\[
2^{\omega(D)}=X^{o(1)},\qquad
\rho(D)\gg_C\frac1{\log\log X}.
\]
\end{lemma}
\begin{proof}
Let $k=\omega(D)$, and write the distinct prime divisors of $D$ as
$q_1<\cdots<q_k$. Since the $j$th smallest prime is at least $j+1$,
$q_j\ge j+1$ for every $j$. Hence
\[
(k+1)!=\prod_{j=1}^k(j+1)
\le \prod_{j=1}^k q_j
\le D
\le X^C.
\]
Taking logarithms gives
\[
C\log X\ge \log((k+1)!).
\]
For $k\ge4$, at least $\lfloor k/2\rfloor$ factors in $(k+1)!$ are at least
$k/2$, so
\[
\log((k+1)!)
\ge \left\lfloor\frac{k}{2}\right\rfloor\log\frac{k}{2}
\gg k\log k.
\]
Thus $k\log k=O_C(\log X)$. If $k\le\sqrt{\log X}$, then already
$k=O(\log X/\log\log X)$. Otherwise
$\log k\ge\tfrac12\log\log X$, and the preceding bound gives
\[
k=O_C\!\left(\frac{\log X}{\log\log X}\right).
\]
Consequently
\[
\log(2^k)=k\log2=O_C\!\left(\frac{\log X}{\log\log X}\right)
=o(\log X),
\]
so $2^k=X^{o(1)}$, proving the first assertion.

For the second, split the prime divisors of $D$ at $y=\log X$. By
\eqref{eq:mertens} and the convergent series $\sum_r r^{-2}$,
\[
\prod_{\substack{r\le y\\r\ \mathrm{prime}}}(1-1/r)
\asymp\frac1{\log y}.
\]
The product over only those primes which divide $D$ is no smaller. Above
$y$, at most $C\log X/\log y$ prime divisors are possible. Their reciprocal
sum is at most $C\log X/(y\log y)=C/\log y$, and hence their contribution
to $-\log\rho(D)$ is $O_C(1/\log y)$. Thus $\rho(D)\gg_C1/\log y$.
Substituting $y=\log X$ proves the claim.
\end{proof}

\subsection{Exact-support saturation and target terminology}
\label{sec:target-terminology}

Local admissibility depends only on prime support. Thus, for a finite set
$T$ of at least two primes, we may say that the support $T$ is locally
admissible at a selection when any, equivalently every, integer with exact
support $T$ satisfies the three support conditions there. Write
\[
\mathcal Q_T=\{m\ge1:P(m)=T\},\qquad
\rad(T)=\prod_{t\in T}t.
\]
Thus $\mathcal Q_T$ is the exact-support queue of $T$, ordered increasingly,
and $\rad(T)$ is the product of the target primes.

\begin{definition}[Saturation]\label{def:saturation}
We say that $T$ is \emph{saturated}, and write
\[
\operatorname{Sat}(T),
\]
if every member of $\mathcal Q_T$ occurs in the Enots--Wolley sequence.
\end{definition}

\begin{lemma}[Exact-support FIFO and finite loss]\label{lem:sat-finite-loss}
Let $T$ be a finite set of at least two primes.
\begin{enumerate}
\item The selected members of $\mathcal Q_T$ occur in increasing numerical
order.
\item $\operatorname{Sat}(T)$ holds if and only if exact support $T$ is
selected infinitely often.
\item If $\operatorname{Sat}(T)$ fails, then there is an index after which
$T$ is never locally admissible.
\end{enumerate}
\end{lemma}
\begin{proof}
Whenever $T$ is locally admissible, every unused member of $\mathcal Q_T$ is
locally admissible. Hence, if the sequence selects a member of
$\mathcal Q_T$, it must select the least unused member; this proves the first
assertion. The same least-unused property shows more: after $k$ selections of
exact support $T$, the first $k$ members of $\mathcal Q_T$ have been selected.
Therefore infinitely many such selections exhaust the queue. Conversely,
saturation plainly implies infinitely many selections of exact support $T$.
This proves the second assertion.

For the third, suppose $\operatorname{Sat}(T)$ fails and choose an omitted
$h\in\mathcal Q_T$. Lemma~\ref{lem:smaller-historical} shows that $h$ is
locally admissible at only finitely many selections. Local admissibility
depends only on prime support, so exact support $T$ itself is locally
admissible only finitely often.
\end{proof}

\begin{definition}[$\neg\operatorname{Sat}$-cutoff]\label{def:sat-cutoff}
Suppose $\operatorname{Sat}(T)$ fails. An index $n_0$ is a
\emph{$\neg\operatorname{Sat}$-cutoff for $T$} if the support $T$ is not
locally admissible at any selection with index $n>n_0$. In particular, no
later selected term has exact support $T$.
Lemma~\ref{lem:sat-finite-loss} shows that $\neg\operatorname{Sat}$-cutoffs
exist and that every larger index is again a $\neg\operatorname{Sat}$-cutoff.
\end{definition}

The main argument also uses the following classes relative to a fixed target
set $T$.

\begin{definition}[Target classes and episodes]\label{def:target-classes}
An integer $m$ is \emph{$T$-free} if $P(m)\cap T=\varnothing$,
\emph{$T$-full} if $T\subseteq P(m)$, and \emph{$T$-proper} if
\[
\varnothing\ne P(m)\cap T\ne T.
\]
It is \emph{$T$-covered} if it is full or proper. A \emph{$T$-covered
episode} is a maximal consecutive block of selected terms that are
$T$-covered; such a block may be finite or an infinite terminal tail. When
$T$ is fixed, we usually omit the prefix $T$.

For each index $N$, write
\[
F_N=\#\{i\le N:a_i\text{ is $T$-full}\},\qquad
P_N=\#\{i\le N:a_i\text{ is $T$-proper}\}.
\]
\end{definition}

\begin{definition}[Proper record]\label{def:proper-record}
A selected term $H=a_N$ is a \emph{proper record} if it is $T$-proper and
exceeds every earlier $T$-proper selected value.
\end{definition}

For example, relative to $T=\{2,3\}$, the integers $35,30,10$ are free,
full, and proper, respectively. A full integer may contain arbitrary primes
outside $T$.

\begin{lemma}[Full-return rule]\label{lem:post-sat-full-return}
Suppose $\operatorname{Sat}(T)$ fails and fix a
$\neg\operatorname{Sat}$-cutoff $n_0$ for $T$. If $B=a_n$ is $T$-free for
some $n\ge n_0$, and the following term $F=a_{n+1}$ is $T$-covered, then $F$
is $T$-full.
\end{lemma}
\begin{proof}
If $F$ were proper, then at the selection of $a_{n+2}$ the support $T$ would
be locally admissible: it overlaps $F$, is disjoint from $B$ because $B$ is
free, and contains a target prime absent from $F$, which supplies novelty.
Since $n+2>n_0$, this contradicts the defining property of the
$\neg\operatorname{Sat}$-cutoff.
\end{proof}

\begin{corollary}[Episode grammar after a $\neg\operatorname{Sat}$-cutoff]
\label{cor:post-sat-episode-grammar}
Suppose $\operatorname{Sat}(T)$ fails and fix a
$\neg\operatorname{Sat}$-cutoff $n_0$ for $T$. Every $T$-covered episode
beginning after $n_0$ starts full and has length at most two. Consequently
every proper term in such an episode is its second term, immediately after a
full starter.
\end{corollary}
\begin{proof}
The starter follows a free term whose index is at least $n_0$, so it is full
by Lemma~\ref{lem:post-sat-full-return}. If the starter is $a_n$, lag-two
disjointness makes $a_{n+2}$ disjoint from every prime of $T$; hence $a_{n+2}$
is free and the episode has length at most two. The final assertion follows.
\end{proof}

\begin{definition}[Eventual prime cover]\label{def:eventual-cover}
A finite set $S$ of primes is an \emph{eventual prime cover} if there is an
index $N_S$ such that
\[
P(a_n)\cap S\ne\varnothing\qquad(n>N_S).
\]
Any such $N_S$ will be called a \emph{cover threshold} for $S$.
\end{definition}

\section{Proof outline}
\label{sec:proof-outline}

To prove Theorem~\ref{thm:main}, it suffices to show that every
exact-support queue $\mathcal Q_T$ is saturated. Suppose, toward a
contradiction, that $\operatorname{Sat}(T)$ fails for some finite prime
set $T$ with $|T|\ge2$, and choose a $\neg\operatorname{Sat}$-cutoff
$n_0$, after which exact support $T$ is never locally admissible. By
Corollary~\ref{cor:post-sat-episode-grammar}, every $T$-covered episode
beginning after $n_0$ starts with a $T$-full term and has length at
most two. Suppose infinitely many such episodes begin, and let $m+1$
be the first such start. Every $T$-proper term selected after $m$ is
then immediately preceded by a distinct full term. Consequently,
with $P_N$ and $F_N$ counting selected proper and full terms through
index $N$, we have
\begin{equation}
P_N\le F_N+O(1).
\label{eq:outline-episode-upper}
\end{equation}

We next establish the existence of arbitrarily large proper-record indices
$N$. At these indices, the number of full terms tends to infinity, and
all but a negligible proportion of them are \emph{good full sources}:
selected full terms from the relevant tail whose prime factorizations
satisfy the conditions needed for the prime-exchange construction.

From each good full source, we construct many smaller proper integers by
replacing its factors from $T$ with a single prime $t\in T$, and replacing
one large prime divisor outside $T$ with another prime. All other outside
factors are retained. These integers are locally admissible at the
selection of the original full term and therefore must have occurred
earlier. Different exchanges may produce the same proper integer. We organize
them into a \emph{weighted prime-exchange multigraph} and compare the
total weight leaving good full sources with the weight entering earlier
proper terms. This controls the overlap and shows that, for some
$\delta_T>0$,
\[
P_N\ge(1+\delta_T)F_N
\]
along arbitrarily large proper records. Since $F_N\to\infty$, this
contradicts \eqref{eq:outline-episode-upper}. Thus, if
$\operatorname{Sat}(T)$ fails, only finitely many $T$-covered episodes can
begin after the cutoff.

Qualitative prime recurrence guarantees that every prime divides
infinitely many terms of the sequence. Consequently,
nonsaturation would imply that every sufficiently late term is
$T$-covered, making $T$ an eventual prime cover. An omitted integer
would then force the existence of two disjoint eventual prime covers.
This is impossible because every sufficiently late prime debut occurs
in a product of two primes, only one of which has appeared previously.
Hence every eligible integer occurs.

\section{Prime recurrence}
\label{sec:quantitative-recurrence}

We first establish qualitative prime recurrence: every prime divides
infinitely many selected terms. We then establish quantitative prime
recurrence, bounding the least unseen prime and the largest selected
term of a finite prefix in terms of its largest selected multiple of
a fixed prime.

\subsection{Qualitative prime recurrence}

The following result is stated as Theorem~6 in \cite{oeis-ew}.

\begin{lemma}[Qualitative prime recurrence]\label{lem:qualitative-prime-recurrence}
Every prime divides infinitely many terms of the Enots--Wolley sequence.
\end{lemma}
\begin{proof}
Suppose a prime $p$ divides only finitely many selected terms. Let
\[
N_p=\max\{j:p\mid a_j\},
\]
and choose a fixed integer $J\ge1$ such that $p^J$ exceeds every selected
multiple of $p$. By Lemma~\ref{lem:unseen-ceiling}, infinitely many primes
debut. Choose a debut $Q>p$ at an index $n>N_p+1$. Then all four terms
\[
A=a_{n-1}\longrightarrow a_n=bQ\longrightarrow
W=a_{n+1}\longrightarrow C=a_{n+2}
\]
are not divisible by $p$. Here $b$ is prime and $b\mid A$, by
Lemma~\ref{lem:debut-pair}.

At the selection of $W$, the integer $pQ$ is unused: the only earlier term
containing $Q$ is $bQ$, with $b\ne p$. It is locally admissible because
$Q$ supplies overlap, $p$ supplies novelty, and both avoid $A$. Since $W$
is not divisible by $p$, greediness gives $W<pQ$. Also, $b\mid A$ excludes
$b$ from $W$, so overlap forces $Q\mid W$. Thus
\[
W=Qv,\qquad 1<v<p<Q.
\]
For $p=2$ this is already impossible. For $p>2$, the integer
$m=p^Jv$ has never been selected: it is a multiple of $p$ exceeding every
selected multiple of $p$. At $C$, it overlaps $W$ through $v$, introduces
$p$, and avoids $bQ$, since $b\nmid v$ and $v<Q$. Thus $m$ is an admissible
candidate at $C$. The relevant prime incidences, with $u$ any prime divisor
of $v$, are
\[
\begin{array}{c|cccc|c}
\text{prime}&A&bQ&W=Qv&C&\text{candidate }m=p^Jv\\ \hline
p&0&0&0&0&1\\
b&1&1&0&0&0\\
Q&0&1&1&0&0\\
u&0&0&1&*&1
\end{array}
\]
where $*$ is unrestricted. Greediness gives
\[
C\le p^Jv<p^{J+1}.
\]
The bound is fixed, whereas infinitely many debut indices give distinct
selected terms $C=a_{n+2}$. This is impossible.
\end{proof}

\subsection{Near-height rescaling}
\label{sec:near-height-rescaling}

For a fixed cofactor $u$, multiplication by powers of $p$ gives the chain
\[
u,\;pu,\;p^2u,\ldots.
\]
We will use the members immediately below and above a specified height.

\begin{definition}[Near-height $p$-rescaling]\label{def:near-height-rescaling}
Let $p$ be a prime, let $u\ge1$ be an integer with $p\nmid u$, and let
$X\ge u$ be real. Define the upward rescaling by
\[
R_p^+(X;u)=\min\{p^ju:j\ge0,\ p^ju>X\}.
\]
If $X>u$, also define the downward rescaling by
\[
R_p^-(X;u)=\max\{p^ju:j\ge0,\ p^ju<X\}.
\]
In both definitions $j$ ranges over integers.
\end{definition}

\begin{lemma}[Rescaling bounds]\label{lem:rescaling-bounds}
For $X>u$, the rescalings satisfy
\begin{equation}
\frac Xp\le R_p^-(X;u)<X<R_p^+(X;u)\le pX.
\label{eq:rescaling-bounds}
\end{equation}
The bounds on $R_p^+(X;u)$ also hold when $X=u$.
The first inequality is strict unless $X/u$ is an integral power of $p$.
The support of $R_p^+(X;u)$ is $P(u)\cup\{p\}$. The same holds for
$R_p^-(X;u)$ when $X>pu$.
\end{lemma}
\begin{proof}
If $R_p^+(X;u)=p^ju$, then $j\ge1$ because $X\ge u$.
Minimality gives $p^{j-1}u\le X$, hence $p^ju\le pX$.
If $R_p^-(X;u)=p^iu$, maximality gives $p^{i+1}u\ge X$, hence
$p^iu\ge X/p$. Equality holds exactly when $X=p^{i+1}u$.
If $X>pu$, then $pu$ itself lies below $X$, so $i\ge1$.
The support identities follow because the corresponding exponents are positive.
\end{proof}

For example, $R_2^-(20;3)=12$ and $R_2^+(20;3)=24$. At height $24$,
the downward value is still $12$, so equality at the left endpoint of
\eqref{eq:rescaling-bounds} can occur. The rescalings are purely arithmetic:
local admissibility and whether the chosen integer has already occurred must
be checked separately in each application.

\subsection{Quantitative prime recurrence}

We now bound the entire prefix from the largest historical multiple of one
fixed prime. Unlike qualitative prime recurrence, this estimate applies to each
finite prefix and controls both its unseen-prime frontier and its largest
selected term.

\begin{lemma}[Quantitative prime recurrence]
\label{lem:quantitative-prime-recurrence}
Fix $N\ge2$ and a prime $p$ which appears among $a_1,\ldots,a_N$. Let $Q_N$
be the least prime which divides none of $a_1,\ldots,a_N$, and set
\[
M_N=\max_{i\le N}a_i,
\qquad
U_{p,N}=\max\{a_i:i\le N,\ p\mid a_i\}.
\]
Then
\begin{align}
\pi(Q_N)&\le\pi(U_{p,N})+pU_{p,N}+3,\label{eq:frontier-rank}\\
Q_N&\ll_p U_{p,N}\log(U_{p,N}+2),\label{eq:frontier-from-U}\\
M_N&\ll_p U_{p,N}^2\log^2(U_{p,N}+2).\label{eq:history-from-U}
\end{align}
\end{lemma}
\begin{proof}
Throughout the proof, the prefix means $a_1,\ldots,a_N$. We split prime
debuts according to whether $Q\le U_{p,N}$ or $Q>U_{p,N}$. The
former contribute at most $\pi(U_{p,N})$ debuts and will be counted at the
end. Notice in particular that any debut where $p$ provides overlap belongs
to this first class: if its debut term is $pQ$, then
$pQ\le U_{p,N}$ by the definition of $U_{p,N}$, and hence
$Q<U_{p,N}$.

Consider therefore a prime $Q>U_{p,N}$ debuting at an index $n\le N-2$,
so its four-term neighborhood is contained in the prefix. By
Lemma~\ref{lem:debut-pair}, write these four consecutive terms as
\[
A=a_{n-1}\longrightarrow a_n=bQ\longrightarrow
W=a_{n+1}\longrightarrow C=a_{n+2}.
\]
We will prove that the term two steps after the debut satisfies
$C\le pU_{p,N}$. Step~4 will show that distinct completed large debuts give
distinct selected values $C=a_{n+2}$. Since there are only $pU_{p,N}$ positive
integers at most $pU_{p,N}$, this will bound the number of such debuts by
$pU_{p,N}$.

\medskip\noindent\textbf{Step 1: $W$ retains $Q$ and is $p$-free.}
The debut term $bQ$ is $p$-free. Indeed, $Q>U_{p,N}\ge p$, so $Q\ne p$;
if $p\mid bQ$, then $b=p$, and the selected $p$-multiple
$pQ>U_{p,N}$ would contradict the definition of $U_{p,N}$. The
carrier $b$ divides $A$, so lag-two disjointness excludes $b$ from $W$.
Since $W$ must overlap $bQ$, we have $Q\mid W$. If $p\mid W$, then
$W\ge pQ>U_{p,N}$, again contradicting the definition of
$U_{p,N}$. Thus $p\nmid W$. Write $W=Qv$. Every noninitial term has at
least two distinct prime divisors, so $v>1$.

The incidences forced so far are
\[
\begin{array}{c|cccc}
\text{prime}
  &A=a_{n-1}& bQ &W=Qv&C\\ \hline
p& *&0&0&*\\
b&1&1&0&0\\
Q&0&1&1&0
\end{array}
\]
where $*$ is unrestricted. The zeros for $b$ and $Q$ in the $C$ column come
from lag-two disjointness with $a_n=bQ$.

\medskip\noindent\textbf{Step 2: the cofactor $v$ of $Q$ in $W=Qv$ satisfies $v\le U_{p,N}$.}
We treat separately the cases $p\nmid A$ and $p\mid A$.
If $p\nmid A$, the integer $pQ$ is unused: the only earlier occurrence of
$Q$ is in $bQ$, and $b\ne p$. It is admissible at $W$, since $Q$ supplies
overlap, $p$ supplies novelty, and both avoid $A$. Equality would put $p$
in $W$, so
\[
W<pQ,\qquad 1<v<p\le U_{p,N}.
\]
If $p\mid A$, then $A\le U_{p,N}$. The case $A=2$ would give $b=p=2$,
contrary to $p\nmid bQ$. Thus $A-1>1$; choose a prime $r\mid A-1$.
It satisfies $r\le U_{p,N}<Q$ and divides neither $A$ nor $bQ$. The
integer $Qr$ is unused, since $Q$ has appeared only in $bQ$ and $r\ne b$.
At $W$ it retains $Q$, introduces $r$, and avoids $A$. Hence
\[
W\le Qr,\qquad 1<v\le r\le U_{p,N}.
\]
Both cases give $1<v\le U_{p,N}$.

\medskip\noindent\textbf{Step 3: a bounded unused $p$-multiple is admissible at $C$.}
Choose a prime divisor $u$ of $v$. It is different from $p,b,Q$: the first
is absent from $W$, the second divides $A$ and so is excluded from $W$,
and the third exceeds $v$. The relevant incidences are
\[
\begin{array}{c|ccc|c}
\text{prime}&A&bQ&W=Qv&\text{candidate at }C\\ \hline
p& *&0&0&1\\
u&0&0&1&1\\
Q&0&1&1&0\\
b&1&1&0&0
\end{array}
\]
where $*$ is unrestricted. Apply upward rescaling at height $U_{p,N}$:
\[
m=R_p^+(U_{p,N};u).
\]
Since $u\le U_{p,N}$ and $u\ne p$,
Lemma~\ref{lem:rescaling-bounds} gives
\[
U_{p,N}<m\le pU_{p,N},\qquad P(m)=\{p,u\}.
\]
This integer is unused by the definition of $U_{p,N}$. It retains $u$
from $W$, introduces $p$, and is disjoint from $bQ$, as required by lag-two
disjointness at the selection of $C$. Greediness gives
$C\le m\le pU_{p,N}$.

\medskip\noindent\textbf{Step 4: completed debuts inject into bounded selected values.}
If the debut occurs at index $n$, its assigned value is $C=a_{n+2}$.
The map
\[
n\longmapsto a_{n+2}
\]
is injective: distinct debut indices give distinct indices $n+2$, and the
sequence never repeats a value. The four-term neighborhoods need not be disjoint.
Since every assigned value is at most $pU_{p,N}$, there are at most
$pU_{p,N}$ completed debuts with $Q>U_{p,N}$.

We now return to the case $Q\le U_{p,N}$ set aside at the start of the
proof. There are at most $\pi(U_{p,N})$ such debut primes. For the
remaining case $Q>U_{p,N}$, Step~4 counts every debut at an index
$n\le N-2$, giving at most $pU_{p,N}$ of them. At most two further
$Q>U_{p,N}$ debuts can occur at indices $N-1$ and $N$, since only those
debuts lack two following selections inside the prefix.

By Lemma~\ref{lem:debut-pair}, prime debuts occur in increasing prime order.
Since $Q_N$ is the least prime that has not debuted by time $N$, the primes
that have debuted by then are exactly the primes below $Q_N$. Finally, $Q_N$
itself has not yet debuted, so including it contributes one further prime.
Therefore
\[
\pi(Q_N)\le \pi(U_{p,N})+pU_{p,N}+2+1
=\pi(U_{p,N})+pU_{p,N}+3,
\]
which is \eqref{eq:frontier-rank}. The prime-index estimate
\eqref{eq:prime-asymptotics} gives \eqref{eq:frontier-from-U};
Lemma~\ref{lem:unseen-ceiling} then gives \eqref{eq:history-from-U}.
\end{proof}

\section{Episodes after a $\neg\operatorname{Sat}$-cutoff and proper records}
\label{sec:full-entry}

\subsection{Proper replacements after a $\neg\operatorname{Sat}$-cutoff}

The first lemma is purely local.

\begin{lemma}[Near-height proper replacement]
\label{lem:protected-replacement}
Let $T$ be a finite set of at least two primes, let $p=\min T$, and let
$A\to B\to f$ be three consecutive selected terms, where $f$ is $T$-full
and $B$ is either $T$-free or $T$-full. Then there is an earlier $T$-proper
value $w$ satisfying
\[
f/p<w<f.
\]
\end{lemma}
\begin{proof}
Choose a prime $h$ as follows:
\[
\begin{array}{c|c|c}
\text{predecessor}&\text{choice of }h&\text{role to preserve}\\ \hline
B\text{ free}&h\in P(B)\cap P(f)&\text{overlap}\\
B\text{ full}&h\in P(f)\setminus P(B)&\text{novelty}.
\end{array}
\]
Overlap supplies $h$ in the first case and novelty supplies it in the second.
In either case $h\notin T$, and $h\nmid A$ because $h\mid f$.
We call such an $h$ a \emph{protected outside prime} for this selection.

Set
\[
w=R_p^-(f;h).
\]
If $q$ is the second-smallest prime of $T$, then the quotient $f/h$ contains
both $p$ and $q$, so $f>ph$ and $f/h$ is not a power of $p$.
Lemma~\ref{lem:rescaling-bounds} therefore gives
\[
f/p<w<f,\qquad P(w)=\{p,h\}.
\]
Thus $w$ is proper and avoids $A$. If $B$ is free, $h$ supplies overlap
and $p$ supplies novelty; if $B$ is full, $p$ supplies overlap and $h$
supplies novelty. The smaller admissible integer $w$ occurred before $f$
by Lemma~\ref{lem:smaller-historical}.
\end{proof}

\begin{corollary}[Full envelope and proper-record supply]
\label{cor:full-envelope}
Let $T$ be a finite set of at least two primes and let $p=\min T$. Suppose
$\operatorname{Sat}(T)$ fails, let $n_0$ be a
$\neg\operatorname{Sat}$-cutoff for $T$, suppose infinitely many $T$-covered
episodes begin after $n_0$, and let $m+1$ be the first such episode start. Let
$H_N$ be the largest proper value among the first $N$ terms, once such a value
exists. Then every full value $f$ selected after $m$ and by $N$ satisfies
\begin{equation}
f<pH_N.
\label{eq:full-envelope}
\end{equation}
Moreover, $H_N$ is unbounded, so there are arbitrarily large proper-record
indices $N>m$.
\end{corollary}
\begin{proof}
By Corollary~\ref{cor:post-sat-episode-grammar}, every full selection after
$m$ has predecessor either free or full. Lemma~\ref{lem:protected-replacement}
therefore supplies an earlier proper value $w$ with $f/p<w<f$. Hence
$w\le H_N$, proving \eqref{eq:full-envelope}.

The infinitely many covered episodes have infinitely many distinct full
starters. Those starters are unbounded, so their proper replacements are
unbounded as well. Thus $H_N$ is unbounded, and every new maximum of $H_N$
occurs at a proper-record index.
\end{proof}

\subsection{Specializing quantitative prime recurrence}

The near-height replacement can now be combined with the general estimate of
Lemma~\ref{lem:quantitative-prime-recurrence} to control the entire prefix from a
proper record.

\begin{corollary}[Proper-record bounds from quantitative prime recurrence]\label{lem:record-history}
Let $T$ be a finite set of at least two primes and let $p=\min T$. Suppose
$\operatorname{Sat}(T)$ fails, let $n_0$ be a
$\neg\operatorname{Sat}$-cutoff for $T$, suppose infinitely many $T$-covered
episodes begin after $n_0$, and let $m+1$ be the first such episode start.
Then at every sufficiently large proper-record index $N>m$, where $H_N=a_N$,
\begin{align}
Q_N&\ll_T H_N\log H_N,\label{eq:record-frontier-bound}\\
M_N&\ll_T H_N^2\log^2 H_N.\label{eq:record-history-bound}
\end{align}
In particular, there is a fixed exponent $C_T$ such that every term in the
entire prefix, and hence every historical term entering a lag-two
disjointness constraint at an earlier full selection, is at most
$H_N^{C_T}$.
\end{corollary}
\begin{proof}
Set
\[
U_{p,N}=\max\{a_i:i\le N,\ p\mid a_i\}.
\]
Every historical multiple of $p$ is full or proper. Proper values are at most
$H_N$, full values selected after $m$ are below $pH_N$ by
\eqref{eq:full-envelope}, and the finitely many values in the prefix through
$m$ are also below $pH_N$ for large $H_N$. Hence $U_{p,N}\le pH_N$.

Equation~\eqref{eq:frontier-from-U} of Lemma~\ref{lem:quantitative-prime-recurrence}
then gives
\[
Q_N\ll_p U_{p,N}\log(U_{p,N}+2)\ll_T H_N\log H_N,
\]
which is \eqref{eq:record-frontier-bound}. Likewise,
\eqref{eq:history-from-U} gives
\[
M_N\ll_p U_{p,N}^2\log^2(U_{p,N}+2)
\ll_T H_N^2\log^2 H_N,
\]
which is \eqref{eq:record-history-bound}. Here $p$ is fixed by $T$.
The polynomial bound follows by absorbing the logarithmic factor into a fixed
power of $H_N$.
\end{proof}

\subsection{A lower bound for the selected full population}

\begin{lemma}[There are enough selected full values]\label{lem:source-population}
Let $T$ be a finite set of at least two primes and let $p=\min T$. Suppose
$\operatorname{Sat}(T)$ fails, let $n_0$ be a
$\neg\operatorname{Sat}$-cutoff for $T$, suppose infinitely many $T$-covered
episodes begin after $n_0$, and let $m+1$ be the first such episode start.
Then at every sufficiently large proper-record index $N>m$, where $H_N=a_N$,
\[
F_N\gg_T\frac{H_N}{\log\log H_N}.
\]
\end{lemma}
\begin{proof}
Because $N>m$ and $H_N=a_N$ is proper, it lies in a covered episode beginning
after $n_0$. Corollary~\ref{cor:post-sat-episode-grammar} says that every
proper term in such an episode is its second term, immediately after a full
starter. Thus the local neighborhood at $H_N$ is
\[
A_{\rm free}\longrightarrow V_{\rm full}\longrightarrow (H_N)_{\rm proper}.
\]
Put $D=\rad(T)AV$. Every integer
\[
z=pu<H_N,\qquad u>1,\qquad (u,D)=1
\]
is proper: $p$ is its only target prime, and $u>1$ has a prime divisor
outside $T$. Its three local admissibility checks at $H_N$ are
\[
\begin{aligned}
p\mid V\text{ and }p\mid pu
   &\quad\text{(overlap)},\\
(pu,A)=1
   &\quad\text{(lag-two avoidance)},\\
u>1\text{ and }(u,V)=1
   &\quad\text{(novelty)}.
\end{aligned}
\]
The second line uses that $A$ is free and $(u,A)=1$; the third supplies a
prime divisor of $u$ absent from $V$. Thus every such $z<H_N$ occurred
earlier by Lemma~\ref{lem:smaller-historical}.

By Corollary~\ref{lem:record-history}, $D$ is bounded by a fixed power of
$H_N$. To count the admissible values of $u$, apply
\eqref{eq:coprime-count} with $z=H_N/p$:
\[
\#\left\{1\le u<\frac {H_N}p:(u,D)=1\right\}
=\frac {H_N}p\rho(D)+O(2^{\omega(D)}).
\]
Our candidate family requires $u>1$, so we discard the possible value $u=1$.
Therefore
\[
\#\left\{1<u<\frac {H_N}p:(u,D)=1\right\}
\ge \frac {H_N}p\rho(D)-O(2^{\omega(D)})-1.
\]
Since $D$ is polynomially bounded in $H_N$, the two conclusions of
Lemma~\ref{lem:polynomial-modulus}, applied with $X=H_N$, give separately
\[
\rho(D)\gg_T\frac1{\log\log H_N},
\qquad
2^{\omega(D)}=H_N^{o(1)}.
\]
Hence
\[
\#\left\{1<u<\frac {H_N}p:(u,D)=1\right\}
\gg_T\frac {H_N}{\log\log H_N}-H_N^{o(1)}
\gg_T\frac {H_N}{\log\log H_N}.
\]
Each such $u$ gives a distinct earlier proper value $pu$. Thus
$P_N\gg_T H_N/\log\log H_N$.

By Corollary~\ref{cor:post-sat-episode-grammar}, every proper term selected
after $m$ is immediately preceded by a distinct full term. Hence
\[
P_N-P_m\le F_N-F_m,
\]
so the fixed prefix contributes only an additive constant. The lower bound
for $P_N$ therefore gives the stated lower bound for $F_N$.
\end{proof}

\section{Removing exceptional full values}
\label{sec:exceptions}

For a finite target set $T$ with $|T|\ge2$, write $p=\min T$. Given a real
parameter $X\to\infty$, set
\begin{equation}
L=\log\log X,\qquad
Y=(\log X)^3,\qquad K=L^4,
\label{eq:cutoffs}
\end{equation}
and, for every positive integer $n$, define
\begin{equation}
\kappa(n)=\#\{r>Y:r\text{ is prime and }r\mid n\}.
\label{eq:kappa}
\end{equation}
The possible full target parts form the set
\begin{equation}
\mathcal C_T=\left\{\prod_{t\in T}t^{e_t}:e_t\ge1\right\}.
\label{eq:full-target-parts}
\end{equation}
Each full value has a unique expression $f=cd_0$ with
$c\in\mathcal C_T$ and $(d_0,\rad(T))=1$. We call $c$ the \emph{full target part} of $f$.

For the application to the EW sequence, suppose $\operatorname{Sat}(T)$
fails, let $n_0$ be a $\neg\operatorname{Sat}$-cutoff for $T$, suppose
infinitely many $T$-covered episodes begin after $n_0$, and let $m+1$ be the
first such episode start. At a sufficiently large proper-record index $N>m$,
put
\begin{equation}
H_N=a_N,\qquad X=pH_N,
\label{eq:record-X}
\end{equation}
and use the cutoffs in \eqref{eq:cutoffs}. Thus $N$ is the sequence index,
whereas $X$ is the numerical size scale attached to that proper record; the
arithmetic counts below are taken over integers up to $X$.

\begin{definition}[Good source]\label{def:good-source}
At this chosen record, let $f=a_i$ be a selected full value with $m<i\le N$,
and let $c$ be its full target part. We call $f$ a \emph{good source} if
\begin{equation}
c\le K,\qquad r^2\nmid f\text{ for every prime }r>Y,
\qquad \kappa(f)\ge4.
\label{eq:good-source}
\end{equation}
\end{definition}

Let $G_N$ count the good sources at this record. The notions $\kappa$ and
``good'' are record-dependent: when $N$ changes, so do $X,L,Y,K$.

The condition $\kappa(f)\ge4$ is chosen because one large outside prime may
later have to be kept fixed in order to preserve a required local incidence.
If $k=\kappa(f)$, then reserving at most one such prime leaves at least
$k-1$ others available. Hence, when $k\ge4$,
\[
\frac{k-1}{k}\ge\frac34.
\]
Thus at least three quarters of the large prime divisors remain available.
The quantitative reason that this amount of reserve is sufficient is given
later, when the relevant comparison is introduced.

\subsection{Large target parts and repeated large primes}

\begin{lemma}[Large target parts are rare]\label{lem:large-target-parts}
Let $T$ be a fixed finite set of at least two primes, let $X\to\infty$, and
define $L,K$ by \eqref{eq:cutoffs}. The number of $T$-full integers
$n\le X$ whose full target part exceeds $K$ is $o_T(X/L^2)$.
\end{lemma}
\begin{proof}
For each full target part $c\in\mathcal C_T$, at most $X/c$
integers $n\le X$ have that target part. Since $c>K$ implies
$c^{-1}\le K^{-3/4}c^{-1/4}$, the number under consideration is at most
\[
X\sum_{\substack{c\in\mathcal C_T\\c>K}}\frac1c
\le XK^{-3/4}\sum_{c\in\mathcal C_T}c^{-1/4}
=XK^{-3/4}\prod_{t\in T}\frac1{t^{1/4}-1}.
\]
The product is a finite product of convergent geometric-series sums,
depending only on the fixed set $T$. Since $K=L^4$, this is
$O_T(X/L^3)=o_T(X/L^2)$.
\end{proof}

\begin{lemma}[Repeated large primes are rare]\label{lem:repeated-large-primes}
Let $X\to\infty$ and define $L,Y$ by \eqref{eq:cutoffs}. The number of
positive integers $n\le X$ divisible by $r^2$ for some prime $r>Y$ is
$o(X/L^2)$.
\end{lemma}
\begin{proof}
A union bound, even over all integers $r>Y$, gives
\[
\#\{n\le X:\exists\text{ prime }r>Y,\ r^2\mid n\}
\le X\sum_{\substack{r>Y\\r\ \mathrm{prime}}}\frac1{r^2}
\ll X/Y=o(X/L^2).
\]
\end{proof}

\subsection{Few large prime divisors: a fourth-moment estimate}

The next lemma gives a fourth-moment estimate for divisibility by primes in a
fixed finite set. A second-moment argument would naturally give only an
$O(X/L)$ exceptional bound, on the same scale as the lower bound supplied by
Lemma~\ref{lem:source-population}. The fourth moment is
used specifically to gain the extra factor of $L$ needed for a negligible
exceptional population.

\begin{lemma}[Fourth-moment transfer]\label{lem:moment-transfer}
Let $\mathcal R$ be a finite set of primes, and let $X\ge1$. Put
\[
\nu(n)=\sum_{r\in\mathcal R}\ind_{r\mid n},\qquad
\mu=\sum_{r\in\mathcal R}\frac1r.
\]
Then
\begin{equation}
\sum_{n\le X}(\nu(n)-\mu)^4
\le X(\mu+3\mu^2)+16|\mathcal R|^4.
\label{eq:moment-transfer}
\end{equation}
\end{lemma}
\begin{proof}
For $r\in\mathcal R$, let $B_r$ be independent Bernoulli variables with
$\Pr(B_r=1)=1/r$, and write $Z_r=B_r-1/r$.
We will prove the estimate in two steps. First, we compare
$\sum_{n\le X}(\nu(n)-\mu)^4$ with
$X\,\mathbb E(\sum_{r\in\mathcal R}Z_r)^4$ term-by-term after expanding both
fourth powers. Then we use independence and centering of the $Z_r$ to bound
the resulting Bernoulli fourth moment.

\medskip\noindent\textbf{(QT) Quadruple transfer.}
Fix $r_1,\ldots,r_4\in\mathcal R$, allowing repetitions, and expand
\[
\prod_{i=1}^4\left(\ind_{r_i\mid n}-\frac1{r_i}\right).
\]
Each of the at most $16$ terms has coefficient of absolute value at most
one. Its product of indicators is the condition $D\mid n$, where $D$ is
the product of the distinct primes present in that term. This condition
holds for $X/D+O(1)$ integers $n\le X$, with error at most one. The
corresponding product of Bernoulli variables has expectation $1/D$.
This also handles repeated primes, since both types of indicators satisfy
$I^j=I$ for $j\ge1$, and the empty product, for which $D=1$. Consequently
\begin{equation}
\left|\sum_{n\le X}\prod_{i=1}^4
 \left(\ind_{r_i\mid n}-\frac1{r_i}\right)
-X\,\mathbb E\prod_{i=1}^4 Z_{r_i}\right|\le16.
\label{eq:quadruple-transfer}
\end{equation}

\medskip\noindent\textbf{(BM) Bernoulli fourth moment.}
Independence and centering leave only fourth powers and paired squares in
the expectation. Since $|Z_r|\le1$ and $\mathbb E Z_r^2\le1/r$,
\begin{equation}
\begin{aligned}
\mathbb E\left(\sum_{r\in\mathcal R}Z_r\right)^4
&=\sum_r\mathbb E Z_r^4
 +6\sum_{r<s}\mathbb E Z_r^2\,\mathbb E Z_s^2\\
&\le\mu+3\mu^2.
\end{aligned}
\label{eq:bernoulli-fourth-moment}
\end{equation}

\medskip\noindent\textbf{Combining the two estimates.}
There are $|\mathcal R|^4$ ordered quadruples. Expanding the fourth powers
and summing the quadruple-transfer estimate \textup{(QT)} over all of them
gives
\[
\sum_{n\le X}(\nu(n)-\mu)^4
\le X\,\mathbb E\left(\sum_{r\in\mathcal R}Z_r\right)^4
     +16|\mathcal R|^4.
\]
Applying the Bernoulli-moment estimate \textup{(BM)} now gives
\[
\sum_{n\le X}(\nu(n)-\mu)^4
\le X(\mu+3\mu^2)+16|\mathcal R|^4,
\]
which is \eqref{eq:moment-transfer}.
\end{proof}

\begin{lemma}[At most three large prime divisors]\label{lem:fourth-moment}
Let $X\to\infty$, define $L,Y$ by \eqref{eq:cutoffs}, and define $\kappa$
by \eqref{eq:kappa}. Then
\[
\#\{n\le X:\kappa(n)\le3\}=O(X/L^2).
\]
\end{lemma}
\begin{proof}
Apply Lemma~\ref{lem:moment-transfer} with $Z=X^{1/8}$ and
$\mathcal R=\{r\text{ prime}:Y<r\le Z\}$, retaining its notation $\nu,\mu$.
For large $X$, $Y<Z$, and Mertens' estimate gives
\begin{equation}
\mu=\log\log Z-\log\log Y+o(1)
=L-\log L+O(1)\sim L.
\label{eq:moment-mean}
\end{equation}
Since $|\mathcal R|\le Z$ and $Z^4=X^{1/2}$, the transfer bound reads
\begin{equation}
\sum_{n\le X}(\nu(n)-\mu)^4
\le X(\mu+3\mu^2)+16X^{1/2}=O(XL^2).
\label{eq:integer-fourth}
\end{equation}
If $\kappa(n)\le3$, then $\nu(n)\le3$. By
\eqref{eq:moment-mean}, each such integer contributes at least
$(\mu-3)^4\gg L^4$ to the left side of \eqref{eq:integer-fourth}.
Dividing by $L^4$ proves the assertion.
\end{proof}

\begin{proposition}[Almost all selected full values are good]
\label{prop:good-sources}
Let $T$ be a finite set of at least two primes and let $p=\min T$. Suppose
$\operatorname{Sat}(T)$ fails, let $n_0$ be a
$\neg\operatorname{Sat}$-cutoff for $T$, suppose infinitely many $T$-covered
episodes begin after $n_0$, and let $m+1$ be the first such episode start. At
every sufficiently large proper-record index $N>m$, put $H_N=a_N$,
$X=pH_N$, and define $L,Y,K,\kappa,G_N$ as above. Then
\[
F_N-G_N=O_T(X/L^2)=o(F_N).
\]
\end{proposition}
\begin{proof}
Every full value selected after $m$ and by index $N$ is below $X=pH_N$ by
\eqref{eq:full-envelope}. Lemmas~\ref{lem:large-target-parts},
\ref{lem:repeated-large-primes}, and \ref{lem:fourth-moment} bound the three
failures of \eqref{eq:good-source} among \emph{all} integers below $X$.
The prefix through $m$ contributes only a fixed number of additional full
values. Finally, Lemma~\ref{lem:source-population} gives
$F_N\gg_T H_N/\log\log H_N$, which is $F_N\gg_T X/L$ because $p$ is fixed by
$T$ and $\log\log H_N=L+O_T(1/\log H_N)$. Thus the exceptional count is
$o(F_N)$.
\end{proof}

\begin{proposition}[Record-scale inputs for the weighted comparison]
\label{prop:record-scale-summary}
Let $T$ be a finite set of at least two primes and let $p=\min T$. Suppose
$\operatorname{Sat}(T)$ fails, let $n_0$ be a
$\neg\operatorname{Sat}$-cutoff for $T$, suppose infinitely many $T$-covered
episodes begin after $n_0$, and let $m+1$ be the first such episode start. At
every sufficiently large proper-record index $N>m$, put $H_N=a_N$ and
$X=pH_N$. Then the following hold simultaneously.
\begin{enumerate}
\item Every selected full value $f=a_i$ with $m<i\le N$ satisfies $f<X$.
\item Writing $A_f=a_{i-2}$, the lag-two predecessor of every such full value satisfies
\[
A_f\le M_N\ll_T H_N^2\log^2 H_N.
\]
Consequently, for every divisor $d$ of such a full value,
\[
A_fd\ll_T H_N^3\log^2 H_N,
\qquad
\omega(A_fd)=O_T(\log X).
\]
\item With $L,Y,K,\kappa,G_N$ defined by \eqref{eq:cutoffs},
\eqref{eq:kappa}, and \eqref{eq:good-source},
\[
F_N\gg_T X/L,
\qquad
F_N-G_N=O_T(X/L^2),
\qquad
G_N=(1-o(1))F_N.
\]
\end{enumerate}
\end{proposition}
\begin{proof}
The first assertion is \eqref{eq:full-envelope}. The bound on $A_f$ is
Corollary~\ref{lem:record-history}; since $d\le f<X=pH_N$ and $p$ is fixed by
$T$, the displayed product bound follows. The elementary inequality
$2^{\omega(n)}\le n$ then gives $\omega(A_fd)=O_T(\log X)$. The population
bounds are Lemma~\ref{lem:source-population} and
Proposition~\ref{prop:good-sources}.
\end{proof}

\section{Admissible prime exchanges}
\label{sec:exchange}

Let $T$ be a finite set of at least two primes, write $p<q$ for its two
smallest members, suppose $\operatorname{Sat}(T)$ fails, and let $n_0$ be a
$\neg\operatorname{Sat}$-cutoff for $T$. Suppose infinitely many $T$-covered
episodes begin after $n_0$, and let $m+1$ be the first such episode start.
Fix a sufficiently large proper-record index $N>m$, put
\[
H_N=a_N,\qquad X=pH_N,
\]
and define $L,Y,K,\kappa$ by \eqref{eq:cutoffs} and \eqref{eq:kappa}. Take
$N$ large enough that $Y>\max T$. Write $F=F_N$, $P=P_N$, and $G=G_N$.
Every occurrence of $\kappa$ below uses the cutoff $Y$ determined by this
chosen record.

\subsection{Constructing the proper alternative}

Fix a good source $f=a_i$, and write
\[
A_f=a_{i-2},\qquad B_f=a_{i-1},
\]
so its local neighborhood is
\[
A_f\longrightarrow B_f\longrightarrow f.
\]
We call $A_f$ the \emph{source lag-two mask}: it is fixed once the source
occurrence $f=a_i$ is fixed, and any replacement compared at that selection
must avoid its prime divisors.

\begin{definition}[Protected outside prime of a source]
\label{def:source-protected-prime}
For a good source $f=a_i$ with predecessor $B_f$, define $h_f$ to be the
smallest prime satisfying
\[
\begin{array}{c|c}
B_f\text{ free}&h_f\in P(B_f)\cap P(f),\\
B_f\text{ full}&h_f\in P(f)\setminus P(B_f).
\end{array}
\]
We call $h_f$ the \emph{protected outside prime of $f$}. In the free case its
existence follows from overlap; in the full case it follows from novelty.
Moreover $h_f\notin T$, and $h_f\nmid A_f$ because $h_f\mid f$. This is the
source-specific choice of the protected outside prime introduced in the proof
of Lemma~\ref{lem:protected-replacement}. Once $f$ is fixed, $h_f$ is fixed
and is used in every admissible prime exchange from $f$.
\end{definition}

\begin{definition}[Admissible prime exchange]\label{def:prime-exchange}
Let $f=a_i$ be a good source with fixed source mask $A_f$ and protected
outside prime $h_f$ from Definition~\ref{def:source-protected-prime}. Let $c$
be the full target part of $f$, and let $r$ be a prime dividing $f$ such
that $r>Y$ and $r\ne h_f$. Define
\begin{equation}
d=\frac{f}{cr},
\qquad\text{so that}\qquad
f=cdr.
\label{eq:source-factorization}
\end{equation}
Since $f$ is good, $r$ occurs once, $(d,\rad(T))=1$, $r\nmid d$, and
$h_f\mid d$.

For this choice of $r$, define the \emph{exchange exclusion product} by
\begin{equation}
E_{f,r}=A_fd.
\label{eq:exchange-exclusion-product}
\end{equation}
For $t\in\{p,q\}$, define the \emph{exchange interval}
\begin{equation}
I_{f,r,t}=\left(\frac{3cr}{4t},\frac{cr}{t}\right).
\label{eq:exchange-interval}
\end{equation}
A prime $\ell$ is an \emph{admissible inserted prime} for $(f,r,t)$ if
\[
\ell\in I_{f,r,t},\qquad \ell\nmid E_{f,r}.
\]
Its corresponding \emph{exchange destination} is
\begin{equation}
w=td\ell.
\label{eq:destination}
\end{equation}
We call the data $(f,r,t,\ell)$ an \emph{admissible prime exchange} from
$f$ to $w$.
\end{definition}

The definition is conditional: it does not yet assert that an admissible
inserted prime $\ell$ exists. Section~\ref{sec:weighted} will show that, at
sufficiently large records, every eligible choice of $r$ and $t$ has many such
primes.

Every full target part contains both $p$ and $q$, so $c/t\ge2$. Thus every
admissible inserted prime satisfies $\ell>3r/2>Y$. In particular, $\ell$ is
not a target prime, differs from $r$, and is absent from $f$ because it
divides neither $c$ nor $d$.

\begin{lemma}[The exchange is a historical proper value]
\label{lem:exchange-admissible}
Let $T$ be a finite set of at least two primes, write $p<q$ for its two
smallest members, suppose $\operatorname{Sat}(T)$ fails, and let $n_0$ be a
$\neg\operatorname{Sat}$-cutoff for $T$. Suppose infinitely many $T$-covered
episodes begin after $n_0$, let $m+1$ be the first such episode start, and fix
a sufficiently large proper-record index $N>m$. Let $f=a_i$ be a good source
at this record, and let $w$ be the destination of an admissible prime exchange
from $f$ in Definition~\ref{def:prime-exchange}. Then $w$ is proper, locally
admissible at the selection of $f$, and satisfies
\[
\frac34f<w<f.
\]
Consequently it occurred before $f$. Moreover,
\begin{equation}
\kappa(w)=\kappa(f).
\label{eq:preserved-kappa}
\end{equation}
\end{lemma}
\begin{proof}
The inequalities follow by multiplying \eqref{eq:exchange-interval} by $td$.
Among primes of $T$, only $t$ divides $w$, and it does so to exponent one;
thus $w$ is proper. The unchanged
factors $t,d$ avoid $A_f$, because they divide $f$; the inserted prime $\ell$
avoids $A_f$ by construction.

If $B_f$ is free, the protected outside prime $h_f\mid d$ supplies overlap
with $B_f$, and $t$ supplies novelty. If $B_f$ is full, $t$ supplies overlap,
and $h_f$ supplies novelty because $h_f\nmid B_f$. Thus $w$ is locally
admissible in both cases, and Lemma~\ref{lem:smaller-historical} makes it
historical.

Finally, the exchange removes exactly one prime above $Y$, namely $r$, and
inserts exactly one different prime above $Y$, namely $\ell$. The removed
prime has exponent one and the inserted prime does not already divide $d$.
All other large prime divisors are unchanged, proving
\eqref{eq:preserved-kappa}.
\end{proof}

The following incidence table collects the two predecessor cases. The two
columns headed $B_f$ are alternatives, not consecutive terms. A $*$ places no
restriction on that incidence.
\begin{center}
\small
\setlength{\tabcolsep}{5pt}
\begin{tabular}{lccccc}
\toprule
Prime role & $A_f$ & $B_f$ free & $B_f$ full & $f$ & $w$\\
\midrule
Chosen target $t$ & 0 & 0 & 1 & 1 & 1\\
Other targets & 0 & 0 & 1 & 1 & 0\\
Protected outside $h_f$ & 0 & 1 & 0 & 1 & 1\\
Removed large $r$ & 0 & $*$ & $*$ & 1 & 0\\
Inserted large $\ell$ & 0 & $*$ & $*$ & 0 & 1\\
\bottomrule
\end{tabular}
\end{center}
The inserted prime is new to the factorization of $f$, not necessarily new
to the sequence.

\subsection{The weighted exchange multigraph}

Our goal is to turn the abundance of good sources into a lower bound for the
number of historical proper values that they force. A single good source may
admit many prime exchanges, but different exchanges can have the same proper
destination. Thus simply counting destinations source by source would
overcount.

We handle these collisions by assigning a positive weight to every admissible
prime exchange. The weights are chosen so that every good source sends a
definite total amount of weight, while every proper destination can receive
only a strictly smaller amount. Since the total outgoing weight equals the
total incoming weight, this comparison forces more proper destinations than
good sources. This strict excess provides the lower bound on proper terms used in the
final contradiction.

\begin{definition}[Weighted exchange multigraph]
\label{def:exchange-multigraph}
At the chosen proper-record index $N$, let $\mathcal G_N$ be the weighted
directed bipartite multigraph defined as follows. Its two vertex classes are
the good sources at this record from Definition~\ref{def:good-source} and the
proper values selected by index $N$.

For every admissible prime exchange $(f,r,t,\ell)$ from
Definition~\ref{def:prime-exchange}, with exchange destination $w$, include one
directed edge $f\to w$ of weight
\begin{equation}
\wt(f,r,t,\ell)=\frac{\log r}{cr\,\kappa(f)}.
\label{eq:edge-weight}
\end{equation}
Distinct admissible prime exchanges give distinct edges, even when they have
the same endpoints. Thus parallel edges are allowed.

For a source $f$, let $\out(f)$ be the total weight of the edges leaving $f$;
for a destination $w$, let $\inn(w)$ be the total weight of the edges entering
$w$. The multigraph $\mathcal G_N$ is finite because there are finitely many
good sources and each exchange interval is bounded.
\end{definition}

When we later bound the incoming weight at a fixed destination $w$,
Equation~\eqref{eq:preserved-kappa} allows the denominator
$\kappa(f)$ in an edge weight to be replaced by $\kappa(w)$. This is what
makes the incoming normalization comparable with the outgoing normalization.

The factors in \eqref{eq:edge-weight} have specific roles. The prime interval
in \eqref{eq:exchange-interval} contains about $cr/(4t\log r)$ primes,
which is canceled by $\log r/(cr)$. A source offers almost $\kappa(f)$
choices of removed prime, whereas a destination has at most $\kappa(w)$
possible inserted primes. Their exact equality is what makes the remaining
normalization work in both directions.

\section{Outgoing and incoming weights}
\label{sec:weighted}

For a finite target set $T$ with at least two primes, write $p<q$ for its two
smallest members. Recall from \eqref{eq:full-target-parts} that $\mathcal C_T$
is the set of full target parts. Define their reciprocal mass by
\begin{equation}
\Sigma_T=\sum_{c\in\mathcal C_T}\frac1c
=\prod_{t\in T}\frac1{t-1}.
\label{eq:target-reciprocal-mass}
\end{equation}
This is a convergent product of geometric series; arbitrary positive target
exponents are included. We will compare
\begin{equation}
\alpha_T=\frac3{16}\left(\frac1p+\frac1q\right),
\qquad \beta_T=\Sigma_T\log(4/3).
\label{eq:transport-constants}
\end{equation}

\subsection{Outgoing weight from a good source}

\begin{lemma}[Uniform outgoing estimate]\label{lem:outgoing}
Let $T$ be a finite set of at least two primes. Suppose
$\operatorname{Sat}(T)$ fails, let $n_0$ be a
$\neg\operatorname{Sat}$-cutoff for $T$, suppose infinitely many $T$-covered
episodes begin after $n_0$, and let $m+1$ be the first such episode start. At
a sufficiently large proper-record index $N>m$, put $H_N=a_N$, $X=pH_N$,
and let $\mathcal G_N$ be the weighted exchange multigraph of
Definition~\ref{def:exchange-multigraph}. Then there is a nonnegative function
$\varepsilon_T(X)\to0$, independent of the source, such that every source
$f$ of $\mathcal G_N$ satisfies
\[
\out(f)\ge\alpha_T-\varepsilon_T(X).
\]
\end{lemma}
\begin{proof}
Fix an eligible removed prime $r$ and a target $t\in\{p,q\}$.

\medskip\noindent\textbf{(PS) Prime supply.}
We first obtain a lower bound for the number of candidate inserted primes
$\ell$ in the exchange interval $I_{f,r,t}$; those excluded by the exchange
exclusion product will be removed in \textup{(EL)} below. We claim that the
following prime-counting asymptotic holds uniformly in the permitted
parameters: there is a nonnegative function $\eta_T(X)\to0$ such that
\begin{equation}
\begin{split}
\#\{\ell\text{ prime}:\ell\in I_{f,r,t}\}
&=\frac{cr}{4t\log r}\bigl(1+\delta_{c,r,t}(X)\bigr),\\
|\delta_{c,r,t}(X)|&\le\eta_T(X),
\end{split}
\label{eq:forward-prime-count}
\end{equation}
uniformly in $c\in\mathcal C_T$, $c\le K$, $r>Y$, and $t\in\{p,q\}$.

To prove this claim, write
\[
E_{\rm PNT}(U)=\sup_{z\ge U}
\left|\frac{\pi(z)\log z}{z}-1\right|.
\]
The prime number theorem gives $E_{\rm PNT}(U)\to0$. Write
\[
I_{f,r,t}=(A,B),
\qquad
A=\frac{3cr}{4t},\qquad B=\frac{cr}{t},
\]
for the two endpoints of the exchange interval.
We have $A\ge3r/2>3Y/2$ and
\[
\log A=\log r+O_T(\log K),\qquad
\log B=\log r+O_T(\log K),
\]
uniformly in the permitted parameters. Since
\[
\frac{\log K}{\log r}\le
\frac{\log K}{\log Y}
=\frac{4\log L}{3L}\longrightarrow0,
\]
applying
\[
\pi(z)=\frac{z}{\log z}\bigl(1+O(E_{\rm PNT}(3Y/2))\bigr)
\]
at $A$ and $B$ gives, uniformly in the permitted parameters,
\[
\pi(B)-\pi(A)
=
\frac{B-A}{\log r}
+
O_T\!\left(
\frac{B}{\log r}
\left(
E_{\rm PNT}(3Y/2)+\frac{\log K}{\log Y}
\right)
\right).
\]
Here we used the preceding uniform comparisons of $\log A$ and $\log B$
with $\log r$. Since $B-A=B/4$, division by $(B-A)/\log r$ shows that
the relative error tends uniformly to zero. Passing from this difference to
the strict interval changes the count by $O(1)$, whose relative contribution
is
\[
O_T\!\left(\frac{\log r}{B}\right)
=O_T\!\left(\frac{\log Y}{Y}\right).
\]
Thus \eqref{eq:forward-prime-count} holds with
\[
\eta_T(X)
=O_T\!\left(E_{\rm PNT}(3Y/2)+\frac{\log K}{\log Y}
+\frac{\log Y}{Y}\right)\longrightarrow0.
\]
The final term absorbs the possible change of at most two primes caused by
the strict endpoints. After multiplication by
the edge weight, the primes supplied by the exchange interval for this
$(r,t)$ have total weight at least
\begin{equation}
W_{\rm PS}(f;r,t)
:=
\frac{1-\eta_T(X)}{4t\,\kappa(f)}.
\tag{PS}
\label{eq:prime-supply-weight}
\end{equation}

\medskip\noindent\textbf{(EL) Exclusion loss.}
We now upper-bound the weight that may be lost when candidate inserted primes
dividing $E_{f,r}$ are excluded.
Proposition~\ref{prop:record-scale-summary}, whose hypotheses are exactly
those in the present lemma, gives
$\omega(A_fd)=O_T(\log X)$ uniformly over all exchanges at the chosen record.
Deleting those primes from the interval therefore removes weight at most
\begin{align}
W_{\rm EL}(f;r,t)
&:=
O_T(\log X)\frac{\log r}{cr\,\kappa(f)}\notag\\
&\le\frac{C_{\mathrm{exc},T}\log X\log Y}{Y\,\kappa(f)}
=\frac{\zeta_T(X)}{\kappa(f)},\qquad
\zeta_T(X)\longrightarrow0.
\tag{EL}
\label{eq:deleted-prime-weight}
\end{align}
for a fixed constant $C_{\mathrm{exc},T}$. We used $c\ge1$ and the decrease of
$(\log u)/u$ for large $u$. Here
$\zeta_T(X)=O_T(L/(\log X)^2)$, independently of the actual prime labels in
the exclusion product.

\medskip\noindent\textbf{(ER) Eligible removals.}
Let $m_f$ be the number of eligible choices of removed prime $r$. Only the
source's protected outside prime $h_f$ from
Definition~\ref{def:source-protected-prime} can be excluded. Hence
\[
\kappa(f)-1\le m_f\le\kappa(f),
\]
and because $\kappa(f)\ge4$,
\begin{equation}
\frac{m_f}{\kappa(f)}
\ge\frac{\kappa(f)-1}{\kappa(f)}
\ge\frac34.
\tag{ER}
\label{eq:eligible-removal-fraction}
\end{equation}

\medskip\noindent\textbf{Combining the three estimates.}
For each eligible pair $(r,t)$, the prime-supply bound \textup{(PS)} gives
candidate weight at least $W_{\rm PS}(f;r,t)$, while the exclusion-loss bound
\textup{(EL)} removes at most $W_{\rm EL}(f;r,t)$. Hence the admissible edges
arising from this pair contribute at least
\[
W_{\rm PS}(f;r,t)-W_{\rm EL}(f;r,t).
\]

First fix an eligible removed prime $r$ and sum over the two target choices.
Using \textup{(PS)} and \textup{(EL)},
\begin{align*}
\sum_{t\in\{p,q\}}
\bigl(W_{\rm PS}(f;r,t)-W_{\rm EL}(f;r,t)\bigr)
&\ge
\sum_{t\in\{p,q\}}
\left[
\frac{1-\eta_T(X)}{4t\,\kappa(f)}
-\frac{\zeta_T(X)}{\kappa(f)}
\right]\\
&=
\frac1{\kappa(f)}
\left[
\frac{1-\eta_T(X)}4
\left(\frac1p+\frac1q\right)
-2\zeta_T(X)
\right].
\end{align*}
The final expression no longer depends on $r$. Since there are exactly $m_f$
eligible choices of removed prime, summing this same lower bound over those
choices gives
\[
\out(f)
\ge
\frac{m_f}{\kappa(f)}
\left[
\frac{1-\eta_T(X)}4
\left(\frac1p+\frac1q\right)
-2\zeta_T(X)
\right].
\]
Set
\[
\varepsilon_T(X)
=\frac{\eta_T(X)}4\left(\frac1p+\frac1q\right)+2\zeta_T(X).
\]
Then
\[
\out(f)
\ge
\frac{m_f}{\kappa(f)}
\left[
\frac14\left(\frac1p+\frac1q\right)
-\varepsilon_T(X)
\right].
\]
Since $X\to\infty$ along these proper-record indices, we may suppose $X$
is large enough that
\[
\frac14\left(\frac1p+\frac1q\right)-\varepsilon_T(X)>0.
\]
Applying the eligible-removal bound \textup{(ER)} then gives
\[
\out(f)
\ge
\frac34\frac14\left(\frac1p+\frac1q\right)
-\varepsilon_T(X)
=
\alpha_T-\varepsilon_T(X).
\]
Since $m_f\le\kappa(f)$, the accumulated error is at most
\[
\frac{m_f}{\kappa(f)}\varepsilon_T(X)
\le\varepsilon_T(X).
\]
\end{proof}

\subsection{A weighted prime sum for the reverse count}

Lemma~\ref{lem:outgoing} showed that every good source sends a definite amount
of weight to historical proper destinations. To turn that outgoing weight into
a lower bound for the number of distinct prior proper values, we must now
control collisions: many different sources may send weight to the same
destination. Thus the next step is to bound the total incoming weight at one
fixed proper destination.

When a destination is fixed and an inserted prime $\ell$ and full target part
$c$ are specified, the possible removed primes $r$ lie in an interval of
multiplicative width $4/3$. Since the edge weight contains the factor
$(\log r)/r$, the weighted prime sum in the following lemma is precisely the
analytic estimate needed for that incoming bound.

\begin{lemma}[Fixed-ratio weighted prime sum]\label{lem:weighted-primes}
As $a\to\infty$,
\begin{equation}
\sum_{\substack{a<r<4a/3\\r\ \mathrm{prime}}}\frac{\log r}{r}
=\log(4/3)+o(1).
\label{eq:weighted-prime-interval}
\end{equation}
The error is uniform for all $a\ge a_0$ as $a_0\to\infty$.
\end{lemma}
\begin{proof}
With weak inclusion at the upper endpoint, partial summation gives
\[
\frac{\vartheta(4a/3)}{4a/3}-\frac{\vartheta(a)}a
+\int_a^{4a/3}\frac{\vartheta(u)}{u^2}\,du.
\]
Write $\vartheta(u)=u(1+\epsilon(u))$, where $\epsilon(u)\to0$ by the
prime number theorem. The two main endpoint terms cancel and the integral
of $1/u$ is $\log(4/3)$. The remaining error is bounded by a constant
multiple of $\sup_{u\ge a}|\epsilon(u)|$. Removing a prime endpoint
changes the answer by $O((\log a)/a)$. Both errors tend to zero uniformly
for $a\ge a_0$.
\end{proof}

\subsection{Incoming weight at one proper value}

\begin{lemma}[Uniform incoming estimate]\label{lem:incoming}
Let $T$ be a finite set of at least two primes. Suppose
$\operatorname{Sat}(T)$ fails, let $n_0$ be a
$\neg\operatorname{Sat}$-cutoff for $T$, suppose infinitely many $T$-covered
episodes begin after $n_0$, and let $m+1$ be the first such episode start. At
a sufficiently large proper-record index $N>m$, put $H_N=a_N$, $X=pH_N$,
and let $\mathcal G_N$ be the weighted exchange multigraph of
Definition~\ref{def:exchange-multigraph}. Then there is a nonnegative function
$\varepsilon'_T(X)\to0$, independent of the destination, such that every
proper value $w$ receiving an edge in $\mathcal G_N$ satisfies
\[
\inn(w)\le\beta_T+\varepsilon'_T(X).
\]
\end{lemma}
\begin{proof}
\medskip\noindent\textbf{(IEI) Incoming-edge injection.}
We claim that, for a fixed destination $w$, there is at most one actual
incoming edge for each triple $(\ell,c,r)$, where $\ell$ is the inserted
prime, $c$ is the full target part of the source, and $r$ is the removed
prime. In other words, suppose two actual exchanges
\[
(f_1,r,t,\ell)\longmapsto w,
\qquad
(f_2,r,t,\ell)\longmapsto w
\]
have the same full target part $c$. Then $f_1=f_2$, so the two incoming edges
are identical.

Indeed, by Lemma~\ref{lem:exchange-admissible}, every actual exchange
destination $w$ has exactly one target prime divisor $t\in\{p,q\}$,
occurring to exponent one. Hence $t$ is determined by $w$. Once $\ell$ is fixed, the
outside cofactor is uniquely determined by
\[
d=\frac{w}{t\ell}.
\]
Once $c$ and $r$ are also fixed, the source is therefore uniquely determined:
\[
f=cdr.
\]
Since the sequence never repeats a value, this numerical value $f$ has a
unique selection index, so its source mask $A_f$ and protected outside prime
$h_f$ are fixed as well. Thus two actual incoming edges at $w$ with the same
triple $(\ell,c,r)$ must be identical.

Finally, solving the exchange-interval inequalities
\eqref{eq:exchange-interval} for the removed prime $r$ gives
\begin{equation}
\frac{t\ell}{c}<r<\frac{4t\ell}{3c}.
\label{eq:reverse-interval}
\end{equation}
Let
\[
\mathcal R_{\ell,c}
=
\left\{r\text{ prime}:
\frac{t\ell}{c}<r<\frac{4t\ell}{3c}\right\}.
\]
Thus $\mathcal R_{\ell,c}$ is the set of possible removed primes in this
range.

\medskip\noindent\textbf{(ES) Enlarged incoming sum.}
For every actual edge, \eqref{eq:preserved-kappa} rewrites its weight as
$\log r/(cr\,\kappa(w))$. By \textup{(IEI)}, we may
upper-bound the incoming weight by allowing every
$r\in\mathcal R_{\ell,c}$. In particular, we drop the requirements that the
reconstructed full value was selected and good, that $r>Y$, and that its
actual exclusion product permitted the exchange. Since all added terms are
positive,
\begin{equation}
\inn(w)
\le\frac1{\kappa(w)}
 \sum_{\substack{\ell\mid w,\ \ell>Y\\\ell\ \mathrm{prime}}}
 \sum_{\substack{c\in\mathcal C_T\\c\le K}}
 \sum_{r\in\mathcal R_{\ell,c}}
 \frac{\log r}{cr}.
\tag{ES}
\label{eq:enlarged-incoming-sum}
\end{equation}
No claim of preserved $\kappa$ is made for the additional virtual
reconstructions.

\medskip\noindent\textbf{(PI) Prime-interval estimate.}
Put $a=t\ell/c$. Its lower bound is
\[
a\ge pY/K\longrightarrow\infty.
\]
Lemma~\ref{lem:weighted-primes} therefore gives, uniformly for every
permitted $\ell,c$,
\begin{equation}
\sum_{r\in\mathcal R_{\ell,c}}
\frac{\log r}{cr}
\le\frac{\log(4/3)+\varepsilon_X}{c},
\qquad \varepsilon_X\longrightarrow0.
\tag{PI}
\label{eq:incoming-prime-interval}
\end{equation}

\medskip\noindent\textbf{Combining the three estimates.}
The incoming-edge injection \textup{(IEI)} justifies the enlarged sum \textup{(ES)}.
Apply the prime-interval estimate \textup{(PI)} to its innermost sum. This
gives
\[
\inn(w)
\le\frac1{\kappa(w)}
 \sum_{\substack{\ell\mid w,\ \ell>Y\\\ell\ \mathrm{prime}}}
 \sum_{\substack{c\in\mathcal C_T\\c\le K}}
 \frac{\log(4/3)+\varepsilon_X}{c}.
\]
Summing over $c$ costs at most the convergent reciprocal mass $\Sigma_T$,
while there are at most $\kappa(w)$ choices of $\ell$. Hence
\[
\inn(w)
\le \Sigma_T\bigl(\log(4/3)+\varepsilon_X\bigr)
=\beta_T+o(1).
\]
Taking $\varepsilon'_T(X)=\Sigma_T\varepsilon_X$ proves the bound.
The error was uniform before either sum: the $c$-sum costs the convergent
reciprocal mass, while the $\ell$-sum cancels $\kappa(w)$.
\end{proof}

\subsection{The strict inequality and episode exclusion after a
$\neg\operatorname{Sat}$-cutoff}

\begin{lemma}[The outgoing constant exceeds the incoming constant]
\label{lem:constant-margin}
Let $T$ be a finite set of at least two primes. Then
\[
\alpha_T>\beta_T.
\]
\end{lemma}
\begin{proof}
Let $p<q$ be the two smallest members of $T$. Every additional target prime
contributes a factor at most one to \eqref{eq:target-reciprocal-mass}, so
\[
\Sigma_T\le\frac1{(p-1)(q-1)}.
\]
Also, since $p\ge2$ and $q\ge3$,
\begin{align*}
(p-1)(q-1)\left(\frac1p+\frac1q\right)
&=\left(1-\frac1p\right)(q-1)
 +\left(1-\frac1q\right)(p-1)\\
&\ge 1+\frac23=\frac53.
\end{align*}
Using \eqref{eq:transport-constants},
\[
\frac{\alpha_T}{\beta_T}
\ge
\frac5{16\log(4/3)}.
\]
Finally,
\[
\exp(5/16)
>
1+\frac5{16}+\frac12\left(\frac5{16}\right)^2
>
\frac43,
\]
so $\log(4/3)<5/16$. Hence $\alpha_T/\beta_T>1$, which is equivalent to
$\alpha_T>\beta_T$.
\end{proof}

\begin{theorem}[Episode exclusion after a $\neg\operatorname{Sat}$-cutoff]
\label{thm:full-entry}
Suppose $\operatorname{Sat}(T)$ fails. For every
$\neg\operatorname{Sat}$-cutoff $n_0$ for $T$, only finitely many $T$-covered
episodes begin after $n_0$.
\end{theorem}
\begin{proof}
Suppose, toward a contradiction, that $\operatorname{Sat}(T)$ fails, that
$n_0$ is a $\neg\operatorname{Sat}$-cutoff for $T$, and that infinitely many
$T$-covered episodes begin after $n_0$. Let $m+1$ be the first such episode
start. By Corollary~\ref{cor:post-sat-episode-grammar}, every proper term
selected after $m$ is immediately preceded by a distinct full term. Hence, for
$N\ge m$,
\[
P_N-P_m\le F_N-F_m,
\]
and therefore $P_N\le F_N+C_0$ with the fixed constant $C_0=P_m$.
Corollary~\ref{cor:full-envelope} gives arbitrarily large proper-record
indices $N>m$.

At such an index put $H_N=a_N$, $X=pH_N$, let $\mathcal G_N$ be the weighted
exchange multigraph of Definition~\ref{def:exchange-multigraph}, and write
$F=F_N$, $P=P_N$, and $G=G_N$. Let $W_N$ be the total edge weight of $\mathcal G_N$. Summing outgoing
weights over its source vertices and incoming weights over its destination
vertices gives the exact identity
\[
W_N
=
\sum_{\substack{f\text{ source of }\mathcal G_N}}\out(f)
=
\sum_{\substack{w\text{ destination of }\mathcal G_N}}\inn(w).
\]
Lemmas~\ref{lem:outgoing} and \ref{lem:incoming}, with their uniform
errors, therefore give
\[
(\alpha_T-\varepsilon_T(X))G
\le W_N\le
(\beta_T+\varepsilon'_T(X))P.
\]
Equivalently,
\[
(\alpha_T-o(1))G\le(\beta_T+o(1))P.
\]
Proposition~\ref{prop:good-sources} gives $G=(1-o(1))F$. Therefore
\begin{equation}
P\ge\left(\frac{\alpha_T}{\beta_T}-o(1)\right)F.
\label{eq:transport-expansion}
\end{equation}
By Lemma~\ref{lem:constant-margin}, $\alpha_T/\beta_T>1$. Since $T$ is fixed,
set
\[
\delta_T=\frac12\left(\frac{\alpha_T}{\beta_T}-1\right)>0.
\]
For all sufficiently large proper-record indices, \eqref{eq:transport-expansion}
therefore gives
\[
P\ge(1+\delta_T)F.
\]
Together with $P\le F+C_0$, this implies
\[
\delta_TF\le C_0.
\]
But $F\to\infty$ along the same indices by
Lemma~\ref{lem:source-population}, a contradiction.
\end{proof}

\section{From nonsaturation to surjectivity}
\label{sec:surjectivity}

The weighted argument has a direct consequence for every nonsaturated exact
support.

\begin{lemma}[Nonsaturation forces an eventual cover]
\label{lem:nonsaturation-cover}
Let $T$ be a finite set of at least two primes. If
$\operatorname{Sat}(T)$ fails, then $T$ is an eventual prime cover.
\end{lemma}
\begin{proof}
Fix a $\neg\operatorname{Sat}$-cutoff for $T$. By qualitative prime
recurrence (Lemma~\ref{lem:qualitative-prime-recurrence}), every
prime in $T$ divides infinitely many selected terms, so there are infinitely
many $T$-covered terms. If there were also infinitely many $T$-free terms,
then there would be infinitely many $T$-covered episodes, and therefore
infinitely many covered episodes beginning after the
$\neg\operatorname{Sat}$-cutoff. This contradicts
Theorem~\ref{thm:full-entry}. Hence only finitely many free terms occur, which
is exactly the eventual-cover property.
\end{proof}

\begin{lemma}[Disjoint eventual covers are impossible]
\label{lem:disjoint-covers}
Two disjoint finite prime sets cannot both be eventual prime covers.
\end{lemma}
\begin{proof}
Suppose disjoint finite sets $S,T$ are both eventual prime covers. Choose
cover thresholds $N_S,N_T$. By Lemma~\ref{lem:unseen-ceiling}, infinitely many
primes debut. Choose a debut prime $Q\notin S\cup T$ whose debut index exceeds
$\max(N_S,N_T)$. Its selected term is $bQ$ for one old prime $b$, by
Lemma~\ref{lem:debut-pair}. To meet $S$ this term requires $b\in S$, since
$Q\notin S$; to meet $T$ it similarly requires $b\in T$. This contradicts
$S\cap T=\varnothing$.
\end{proof}

\begin{proof}[Proof of Theorem~\ref{thm:main}]
Suppose an eligible integer $m$ is omitted. Then
\[
\neg\operatorname{Sat}(P(m)).
\]
By Lemma~\ref{lem:nonsaturation-cover}, the finite set
\[
S=P(m)
\]
is an eventual prime cover. Choose a cover threshold $N_S$ for $S$.

Only finitely many primes divide the first $N_S$ terms. Choose distinct primes
$r,s$ outside $S$ and outside all of those supports. The eligible integer
\[
u=rs
\]
cannot occur by time $N_S$, because neither of its prime divisors has appeared
there. It cannot occur afterward either, because its support $\{r,s\}$ is
disjoint from the eventual cover $S$. Hence $u$ is omitted, so
$\operatorname{Sat}(\{r,s\})$ fails.

Lemma~\ref{lem:nonsaturation-cover} now makes $\{r,s\}$ an eventual prime
cover. This cover is disjoint from $S$, contradicting
Lemma~\ref{lem:disjoint-covers}. Therefore every eligible integer occurs.

The seeds $1,2$ occur, and the overlap and novelty conditions require every
later term to have at least two distinct prime divisors. The terms are
distinct by construction, proving exactly the claimed range.
\end{proof}

\section*{Declaration of AI Use}
OpenAI's ChatGPT was used extensively in both the mathematical research and
the preparation of this manuscript. Its contributions included exploring and
developing proof strategies, lemmas, and intermediate arguments; identifying
potential gaps and counterexamples; assisting with computational experiments
and code; and drafting and revising substantial portions of the \LaTeX{}
manuscript. The author directed and reviewed this work, revised AI-generated
material, and takes full responsibility for the mathematical claims and the
submitted manuscript.
\clearpage
\bibliographystyle{plain}
\bibliography{references}
\end{document}